\documentclass[11pt]{amsart}
\usepackage{mathpazo}  
\usepackage{geometry}              % See geometry.pdf to learn the layout options. There are lots.
\usepackage{graphicx}
\usepackage[parfill]{parskip} 
\usepackage{amssymb}
\usepackage[all]{xy}
\usepackage{color}
\usepackage{tikz}
\usepackage{dsfont}
\definecolor{dmagenta}{rgb}{.5,0,.5} 
\definecolor{dred}{rgb}{.5,0,0} 
\definecolor{dgreen}{rgb}{0,.5,0} 
\definecolor{blue}{rgb}{0,0,0.5} 
\definecolor{black}{rgb}{0,0,0} 
\definecolor{vdgreen}{rgb}{0,.3,0} 
\definecolor{vdred}{rgb}{.3,0,0} 
\definecolor{red}{rgb}{1,0,0} 

\newcommand{\Q}{\F}

\newcommand{\F}{{\mathds{k}}}
\newcommand{\Z}{{\mathbb{Z}}}

\newcommand{\hairy}{\mathcal H}  %Notation for the hairy graph complex

\newcommand{\ext}{\bigwedge\nolimits}

\DeclareMathOperator{\Tr}{Tr}  % Trace map
\DeclareMathOperator{\SP}{Sp}  % symplectic group
\DeclareMathOperator{\Out}{Out} % Out
\newcommand{\Sym}[1]{\mathbb S_{#1}}
\newcommand{\sym}{\Sym}
\DeclareMathOperator{\GL}{GL}  % general linear group
\DeclareMathOperator{\Aut}{Aut}  % Aut
\DeclareMathOperator{\Mod}{Mod}  % Mapping class groups
\newcommand{\SF}[1]{{\mathbb S}_{#1}}  % Schur functor
\newcommand{\cT}{\mathcal T}

\newcommand{\la}{\langle}
\newcommand{\ra}{\rangle}

\newcommand{\newtrace}{\Tr^{\mathrm C}}
\newcommand{\gspace}{\Omega }

\newtheorem{proposition}{Proposition}[section]
\newtheorem{theorem}[proposition]{Theorem}

\newtheorem{corollary}[proposition]{Corollary}
\theoremstyle{remark}

\theoremstyle{definition}
\newtheorem{definition}[proposition]{Definition}
\newtheorem{remark}[proposition]{Remark}
\newtheorem{conjecture}[proposition]{Conjecture}

\newtheoremstyle{red}{3pt}{3pt}{\color{red}}{}{\itshape}{.}{.5em}{}
\theoremstyle{red}

\title{The Johnson Cokernel and the Enomoto-Satoh invariant
}
\author{Jim Conant}
\begin{document}
\begin{abstract}
We study the cokernel of the Johnson homomorphism for the mapping class group of a surface with one boundary component. A graphical trace map simultaneously generalizing trace maps of Enomoto-Satoh  and Conant-Kassabov-Vogtmann is given, and using technology from the author's work with Kassabov and Vogtmann, this is is shown to detect a large family of representations which vastly generalizes series due to Morita and Enomoto-Satoh. The Enomoto-Satoh trace is the rank $1$ part of the new trace. The rank 2 part is also investigated.
\end{abstract}
\maketitle
\section{Introduction}
The Johnson homomorphism is an injective Lie algebra homomorphism $\tau\colon \mathsf J\to \mathsf D(H)$ \cite{johnson,morita}, where $\mathsf J$ is the associated graded Lie algebra coming from the Johnson filtration of the mapping class group $\Mod(g,1)$ and $\mathsf{D}(H)=\mathsf{D}(H_1(\Sigma_{g,1};\F))$ is a Lie algebra of ``symplectic derivations" of the free Lie algebra $\mathsf{L}(H)$.
It is an isomorphism in order $1$: $\mathsf J_1\cong \mathsf D_1(H)\cong \ext^3 H$, and in fact a theorem of Hain \cite{Hain} says that $\tau(\mathsf J)$ is generated as a Lie algebra by the order $1$ part $\ext^3 H$.
 In general, $\tau$ is not surjective and the Johnson cokernel $\mathsf C_s=\mathsf D_s(H)/ \tau (\mathsf J_s)$ is an interesting $\SP(H)$-module. (See Figure~\ref{fig:cokernel} for the known decomposition in low degrees.)
 
In Morita's 1999 survey article \cite{MoritaProspect}, he listed a series of problems that he felt were important future directions of research in the mapping class group. One of these problems was to determine exactly how $\mathsf J$ includes into $\mathsf D(H)$ as an $\SP$-module, and in particular, to characterize the cokernel of the Johnson homomorphism. One source of interest in this problem comes from number theory. Nakamura \cite{nakamura} showed that certain obstructions coming from the Galois group $\mathrm{Gal}(\overline{\mathbb Q}/\mathbb Q)$ appear in the cokernel in even orders $2k$. Deligne's motivic conjecture would imply that these obstructions appear with multiplicity given by the degree $k$ part of the free graded Lie algebra $\mathsf{L}(\sigma_3,\sigma_5,\sigma_7,\cdots)$ with one generator in each odd degree $\geq 3$. All of the representations coming from this so-called ``Galois obstruction" appear as the trivial $\SP$-representation $[0]_{\SP}$, giving an infinite family of cokernel obstructions. Morita \cite{morita} showed that representations $[k]_{\SP}$ appear in the cokernel for all odd $k\geq 3$, and more recently Enomoto and Satoh \cite{ES} showed that representations $[1^{4m+1}]_{\SP}$ appear in the cokernel as well.

In this paper, we introduce a new invariant for detecting the cokernel
$$\Tr^{\mathrm{C}}\colon \mathsf C_s \to \bigoplus_{r\geq 1}\gspace_{s+2-2r,r}(H)$$
which simultaneously generalizes the construction of Enomoto-Satoh \cite{ES} and of Conant-Kassabov-Vogtmann \cite{CKV}. (The superscript ``C" stands for ``cokernel.")
The space $\gspace_{s+2-2r,r}(H)$ is defined as a quotient of the dimension $1$ part of the hairy graph complex \cite{CKV} by certain relators, shown on the right of Figure~\ref{boundaries}. The set of relations is large enough so that $\Tr^{\mathrm C}$ vanishes on iterated brackets of order $1$ elements, but not so large as to project all the way down to the first homology of the hairy graph complex. The two indices $s+2-2r$ and $r$ refer to the \emph{number of hairs} and \emph{rank} of the graph, respectively.

The $r=1$ part $\gspace_{s,1}(H)$  is isomorphic to $[H^{\otimes s}]_{D_{2s}}$ and $\Tr^{\mathrm{C}}$ projects to the Enomoto-Satoh trace $\Tr^{\mathrm{ES}}\colon \mathsf{C}_s\to [H^{\otimes s}]_{D_{2s}}$. (Although their trace takes values in $[H^{\otimes s}]_{\Z_{s}}$, it possesses an extra $\Z_2$-symmetry.) 

\begin{figure}
 \begin{enumerate}
\item[] $\mathsf C_1=\mathsf C_2=0$
\item[] $\mathsf C_3= [3]_{\SP}$
\item[] $\mathsf C_4=[21^2]_{\SP}\oplus[2]_{\SP}$
\item[] $
\mathsf C_5=[5]_{\SP}\oplus[32]_{\SP}\oplus[2^21]_{\SP}\oplus[1^5]_{\SP}\oplus
2[21]_{\SP}\oplus2[1^3]_{\SP}\oplus
2[1]_{\SP}
$
\item[] $\mathsf C_6=2[41^2]_{\SP} \oplus [3^2]_{\SP} \oplus [321]_{\SP} \oplus [31^3]_{\SP} \oplus [2^21^2]_{\SP}
\oplus2[4]_{\SP} \oplus 3[31]_{\SP} \oplus 3[2^2]_{\SP} \oplus 3[21^2]_{\SP} \oplus 2[1^4]_{\SP} \oplus [2]_{\SP} \oplus 5[1^2]_{\SP} \oplus 3[0]_{\SP}$
 \end{enumerate}
 
\caption{The Johnson cokernel in low orders. \cite{MSS3}}\label{fig:cokernel}
 \end{figure}

 Let $H^{\la s\ra}\subset H^{\otimes s}$ be the intersection of the kernels of all the pairwise contractions $H^{\otimes s}\to H^{\otimes (s-2)}$.  Then there is a projection $\pi\colon\gspace_{s+2-2r,r}(H)\to \gspace_{s+2-2r,r}\la H\ra $ where the latter space is defined by ``taking coefficients in $H^{\la s+2-2r\ra}$."  
A theorem of \cite{CKV2} implies that the composition $\pi\circ\Tr^{\mathrm C}$ is onto.
  Considering the case $r=1$ gives us the following theorem.

\begin{theorem}
There is an epimorphism $\mathsf C_s\twoheadrightarrow [H^{\la s\ra}]_{D_{2s}}$ where the dihedral group acts on $H^{\otimes s}$ in the natural way, twisted by the nontrivial $\Z_2$ representation when $s$ is even.
\end{theorem}
This theorem vastly generalizes the known results for size $s$ representations in $\mathsf C_s$, which essentially consist of the two series due to Morita and Enomoto-Satoh described above, and of low order calculations.
We show in Theorem~\ref{thm:ESM} that both infinite  series are contained in $[H^{\la s\ra}]_{D_{2s}}$.  Comparing to computer calculations by Morita-Sakasai-Suzuki \cite{MSS3} shows that $[H^{\la s\ra}]_{D_{2s}}$ contains all size $s$ representations in $\mathsf C_s$ for $s\leq 6$, which is as far as calculated. A heuristic argument (see section \ref{reptheory}) shows that ``most" representations $[\lambda]_{\SP}$ appear in $[H^{\la s\ra}]_{D_{2s}}$.
In Theorems~\ref{thm:p} and \ref{thm:2p} explicit large infinite families of representations are constructed.

Turning now to $r\geq 2$, in a future paper \cite{CK}, we will show that there is an epimorphism $\gspace_{s+2-2r,r}(H)\twoheadrightarrow H^1(\Out(F_r);M_{s+2-2r,r})$ where $M_{s+2-2r,r}$ is a certain $\Out(F_r)$-module constructed from the tensor algebra $T(H)$. 
In the present paper, we give a description of $\gspace_{s-2,2}(H)$ in terms of generators and relations, and using this presentation to do computer calculations (see Theorem \ref{thm:omega2}), we find $\gspace_{s-2,2}\la H\ra$ for $s\leq 8$.

We finish the introduction by comparing our construction to the abelianization. Letting $\mathsf D^{\mathrm{ab}}_s(H)$ be the order $s$ part of the abelianization of $\mathsf D^+(H)$, Hain's theorem implies that $\mathsf C_s\twoheadrightarrow \mathsf D^{\mathrm{ab}}_s(H)$ for $s>1$. So the abelianization detects cokernel elements. 
A theorem of \cite{CKV2} implies that $\gspace_{s+2r-2,r}\la H\ra$ projects onto the rank $r$ part of the abelianization $\mathsf D^{\mathrm {ab}}_{s}(H)$, with the rank defined in the sense of \cite{CKV,CKV2}. The rank $1$ part of the abelianization consists of Morita's $[2m+1]_{\SP}$ for $m>1$, which does indeed appear in $\gspace_{2m+1,1}\la H\ra$ as noted above. The rank $2$ part of the abelianization consists of the following representations \cite{CKV2}:
 for all $k>\ell\geq 0$
  $$[2k,2\ell]_{\SP}\otimes \mathcal S_{2k-2\ell+2}\subset \mathsf D^{\mathrm{ab}}_{2k+2\ell+2}(H)$$ and $$[2k+1,2\ell+1]_{\SP}\otimes\mathcal M_{2k-2\ell+2}\subset \mathsf D^{\mathrm{ab}}_{2k+2\ell+4}(H),$$
where $\mathcal S_{w}$ and $\mathcal M_w$ are the vector spaces of weight $w$ cusp forms and modular forms respectively. Hence, these are detected by $\oplus_s\gspace_{s-2,2}\la H\ra$. 
However $\oplus_s\gspace_{s-2,2}\la H\ra$ contains a lot more, as the calculations of Theorem \ref{thm:omega2} indicate. In fact, we show in \cite{CK} that $\oplus_s\gspace_{s-2,2}(H)$ surjects onto
\begin{align*}
& \bigoplus_{k>\ell\geq 0} \mathcal S_{2k-2\ell+2}\otimes \left(\frac{\SF{(2k,2\ell)}(\mathsf{L})}{\mathrm{ad}(\mathsf{L})\cdot\SF{(2k,2\ell)}(\mathsf{L})}\right)\oplus\\
 & \bigoplus_{k>\ell\geq 0} \mathcal M_{2k-2\ell+2}\otimes \left(\frac{\SF{(2k+1,2\ell+1)}(\mathsf{L})}{\mathrm{ad}(\mathsf{L})\cdot\SF{(2k+1,2\ell+1)}(\mathsf{L})}\right)
\end{align*}
where $\mathsf L=\mathsf L(H)$ is the free Lie algebra on $H$ and $\mathrm{ad}(\mathsf{L})$ is the adjoint action of $\mathsf L$ on the Schur functor $\SF{\lambda}(\mathsf{L})$.
 The appearance of modular forms and the free Lie algebra $\mathsf L(H)$ in the Johnson cokernel provides yet another connection to number theory which is not yet fully understood.

{\bf Acknowledgements:} I'd like to thank the organizers Nariya Kawazumi and Takuya Sakasai of the 2013 workshop at the University of Tokyo on the Johnson Homomorphism for providing a fertile research environment. I'd also like to thank Naoya Enomoto for his talk and subsequent discussion, and Martin Kassabov for insightful comments and discussion. Finally I thank Jetsun Drolma for her inspiration in Yoyogi park. 

\section{Basic definitions}\label{sec2}
Fix a base field $\F$ of characteristic $0$. Let $\Sigma_{g,1}$ be a surface of genus $g$ with one boundary component. Throughout the paper we let $H=H_1(\Sigma_{g,1};\F)$, which is a symplectic vector space. We let $\la\cdot,\cdot\ra$ denote the symplectic form, and let $p_1,\ldots,p_g,q_1,\ldots, q_g$ be a symplectic basis. We say $\la v,w\ra$ is the \emph{contraction} of $v$ and $w$.
Let $\mathbb S_s$ be the symmetric group on $s$ letters and for the groups $G\in\{\SP(H),\GL(H),\mathbb S_s\}$, let $[\lambda]_{G}$ be the irreducible representation of $G$ corresponding to $\lambda$. 

We begin by defining the relevant Lie algebra which is the target of the Johnson homomorphism.

\begin{definition}
 Let $\mathsf{L}_k(H)$ be the degree $k$ part of the free Lie algebra on $H$. Define $\mathsf D_s(H)$ to be the kernel of the bracketing map $H\otimes \mathsf  L_{s+1}(H)\to \mathsf L_{s+2}(H)$. Let $\mathsf D(H)=\bigoplus_{s=0}^\infty \mathsf D_s(H)$ and $\mathsf D^+(H)=\bigoplus_{s\geq 1} \mathsf D_s(H)$. We refer to $s$ as the \emph{order} of an element of $\mathsf D(H)$.
\end{definition}
$H\otimes \mathsf L(H)$ is canonically isomorphic via the symplectic form to $H^*\otimes \mathsf L(H)$ which is isomorphic to the space of derivations $\mathsf {Der} (\mathsf L(H))$. Under this identification, the subspace $\mathsf D(H)$ is identified with $\mathsf {Der}_\omega(\mathsf L(H))=\{X\in \mathsf {Der}(H)\,|\, X\omega =0\}$ where $\omega=\sum [p_i,q_i]$. Thus $\mathsf D(H)$ is a Lie algebra with bracket coming from $\mathsf {Der}_\omega(H)$. 

There is another beautiful interpretation of this Lie algebra in terms of trees:
 \begin{definition}
  Let $\cT(H)$ be the vector space of unitrivalent trees where the univalent vertices are labeled by elements of $H$ and the trivalent vertices each have a specified cyclic order of incident half-edges, modulo the  standard AS, IHX and multilinearity relations.(See Figure~\ref{fig:multilin} for the multilinearity relation.) Let $\cT_k(H)$ be the part with $k$ trivalent vertices. Define a Lie bracket on $\cT(H)$ as follows. Given two labeled trees $t_1,t_2$, the bracket $[t_1,t_2]$ is defined by summing over joining a univalent vertex from $t_1$ to one from $t_2$, multiplying by the contraction of the labels.
\end{definition}
These two spaces $\mathsf{D}_s(H)$ and $\cT_s(H)$ are connected by a map $\eta_s\colon \cT_s(H)\to H\otimes \mathsf L_{s+1}(H)$ defined   
by $\eta_s(t)=\sum_x \ell(x)\otimes t_x$ where the sum runs over univalent vertices $x$, $\ell(x)\in H$ is the label of $x$, and $t_x$ is the element of $\mathsf L_{s+1}(H)$ represented by the labeled rooted tree formed by removing the label from $x$ and regarding $x$ as the root. The image of $\eta_s$ is contained in $\mathsf D_s(H)$ and gives an isomorphism $\cT_s(H)\to \mathsf D_s(H)$ in this characteristic $0$ case \cite{Levine}.

\begin{figure}
$$
\begin{minipage}{2cm}
\includegraphics[width=2cm]{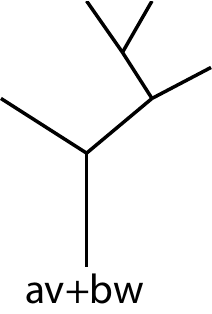}
\end{minipage}
=
a
\begin{minipage}{2cm}
\includegraphics[width=2cm]{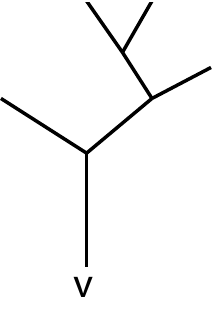}
\end{minipage}
+b
\begin{minipage}{2.7cm}
\includegraphics[width=2.7cm]{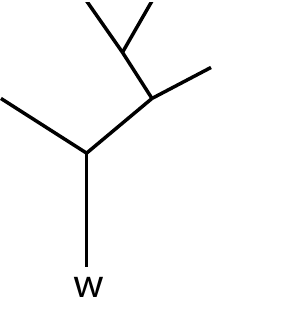}
\end{minipage}
$$
\caption{Multilinearity relation in $\cT(H)$. Here $a,b\in\F$, $v,w\in V$}\label{fig:multilin}
\end{figure}

%Suppose $H$ is a symplectic vector space with symplectic form $\la\cdot,\cdot\ra$ and a symplectic basis $p_1,\ldots,p_n,q_1,\ldots, q_n$. The expression $\la v,w\ra$ is called the contraction of $v$ and $w$. The contraction of $p_i$ and $q_i$ is $1$, whereas the contraction of $q_i$ with $p_i$ is $-1$. All other contractions of basis elements are zero.

Now that we understand the target of the Johnson homomorphism, we review the construction of the homomorphism itself.
 Let $F=\pi_1(\Sigma_{g,1})$ be a free group on $2g$ generators and given a group $G$, let $G_k$ denote the $k$th term of the lower central series: $G_1=G$ and $G_{k+1}=[G,G_k]$. The Johnson filtration  $$\Mod(g,1)= \mathbb J_0 \supset \mathbb J_1\supset \mathbb J_2\cdots $$
of the mapping class group $\Mod(g,1)$ is defined by letting $\mathbb J_s$ be the kernel of the homomorphism $\Mod(g,1)\to \Aut(F/F_{s+1})$. The \emph{associated graded} $\mathsf J_s$ is defined by $\mathsf J_s=\mathbb J_s/\mathbb J_{s+1}\otimes \Q$. (The Johnson filtration is a central series, so that the groups $\mathsf J_k$ are abelian.) Let $\mathsf J=\bigoplus_{s\geq 1} \mathsf J_s$, where we refer to $s$ as the \emph{order} of the element. 

The group commutator on $\Mod(g,1)$ induces a Lie algebra structure on $\mathsf J$.

It is well-known that $\Mod(g,1)\cong \Aut_0(F)$ where $\Aut_0(F)=\{\varphi\in \Aut(F)\,|\, \varphi(\prod_{i=1}^g[p_i,q_i]=[p_i,q_i])$.

\begin{definition}
 The (generalized) Johnson homomorphism $\tau\colon \mathsf J\to \mathsf D^+(H)$ is defined as follows. Let $\varphi\in \mathbb J_s$. Then $\varphi$ induces the identity on $\Aut(F/F_{s+1})$. Hence for every $z\in F$, $z^{-1}\varphi(z)\in F_{s+1}$, and we can project to get an element $[z^-1\varphi(z)]\in F_{s+1}/F_{s+2}\otimes \F\cong \mathsf L_{s+1}(H)$.
 Define a map $\tau(\varphi)\colon H\to  \mathsf L_{s+1}(H)$ via $z\mapsto [z^-1\varphi(z)]$ where $z$ runs over the standard symplectic basis of $H$. By the various identifications, we can regard $\tau(\varphi)$ as being in $\mathsf L\otimes \mathsf L_{s+1}(H)$. 
 The fact that $\varphi$ preserves $\prod_{i=1}^g [p_i,q_i]$ ensures that $\tau(\varphi)\in\mathsf{D}_s(H)\subset \mathsf L\otimes \mathsf L_{s+1}(H)$.
 \end{definition}
 
 \begin{proposition}[Morita]
 The Johnson homomorphism $\tau \colon \mathsf J\to \mathsf D^+(H)$ is an injective homomorphism of Lie algebras.
 \end{proposition}
 
 The main object of study of this paper is \emph{the Johnson cokernel:}
 $$\mathsf C_s=\mathsf D_s(H)/\tau(\mathsf J_s).$$ 
More precisely, we are interested in the stable part of the cokernel and we always assume that $2g=\dim (H)\gg s$.

\section{The construction}
We recall from \cite{CKV} the definition of the hairy Lie graph complex and the trace map. The hairy graph complex $C_k\hairy(H)$ is defined as the vector space with basis given by certain types of decorated graphs modulo certain relations. 

We begin by describing the generators. Start with a union of $k$ unitrivalent trees with specified cyclic orders at each trivalent vertex. Then
 join several pairs of univalent vertices by edges, which  are called \emph{external edges}. (One can think of the trees and added edges as being different colors. We will use the convention that external edges are dashed.) 
  The univalent vertices of the trees that were not paired by edges are each labeled by an element of the vector space $H$. Such a graph is called a \emph{hairy graph}.
Hairy graphs have an \emph{orientation},
 which is defined as a bijection of the trees with the numbers $1$ to $k$ and a direction on each external edge. 
 
 The relations are: 
 \begin{enumerate}
\item IHX within trees,
\item AS within trees, 
\item multilinearity on labels of univalent vertices, 
\item switching an edge's direction gives a minus sign, 
\item renumbering the trees gives the sign of the permutation. 
\end{enumerate}
These last two types of relations explain how changing the decorations of the graph switches the orientation. Informally $C_k\hairy(H)$ is the space you get by joining $k$ elements of $\cT(H)$ by several external edges and giving the resulting object an orientation in the above sense.

The boundary operator $\partial:C_k\hairy(H)\to C_{k-1}\hairy(H)$ is defined on hairy graphs by summing over joining pairs of trees along external edges. The sign and induced orientation are fixed by the convention that contracting a directed edge from tree $1$ to tree $2$ induces the orientation where all edge directions are unchanged, the tree formed by joining tree 1 and 2, is numbered $1$ and all other tree numbers are reduced by $1$.

In \cite{CKV}, we showed that the abelianization $\mathsf D^{\mathrm {ab}}(H)$ embeds in $H_1(\hairy(H))$ via a map which we now define. First, define an operator $T\colon C_k\hairy(H)\to C_k\hairy(H)$  by summing over adding an external edge  to all pairs of univalent vertices of a hairy graph, fixing the direction arbitrarily and multiplying by the contraction of the two labels. Also define a natural inclusion $\iota\colon \ext^k \cT(H)\to C_k\hairy(H)$ by regarding $t_1\wedge\cdots\wedge t_k$ as a union of trees with no external edges. The ordering from the wedge converts to a numbering of the trees as required for the orientaion in $C_k\hairy(H)$. Now we can define the trace map from \cite{CKV}.
\begin{definition}
The trace map $\Tr^{\mathrm{CKV}}\colon \ext^k \cT(H)\to C_k\hairy(H)$ is defined as $\Tr^{\mathrm{CKV}}=\exp(T)\circ\iota$.
\end{definition}
Unpacking the definition, the trace map $\Tr^{\mathrm{CKV}}$ adds several external edges to a hairy graph in all possible unordered ways. In \cite{CKV}, $\Tr^{\mathrm{CKV}}$ is shown to be a chain map, which is injective on homology, so induces an injection from the abelianization to $H_1(\hairy(H))$. 

Now to define $\newtrace$, consider the subspace $S_2\subset C_2\hairy(H)$ consisting of an order $1$ tree (tripod) which is connected by two or three of its hairs to the other tree, or has two of its hairs joined by an edge, and the third edge is connected to the other tree. The other tree may have edges connecting it to itself.

\begin{definition}
The target of $\newtrace$ is defined as $\gspace(H)=C_1\hairy(H)/(\partial(S_2)+\iota(\cT(H)))$.
\end{definition}
The $\iota(\cT(H))$ term is to eliminate graphs without any edges. Notice that by definition $\gspace(H)$ surjects onto the part of $H_1(\mathcal H_H)$ with at least one edge. See Figure~\ref{boundaries} for a depiction of the three types of relations coming from $\partial(S_2)$. The first kind says that an isolated loop is zero. The second kind says that one can slide a hair along an external edge. The third kind is more complicated, but does not appear until there are at least two external edges attached.
\begin{figure}
\begin{enumerate}
\item $$\partial \begin{minipage}{2cm}\includegraphics[width=2cm]{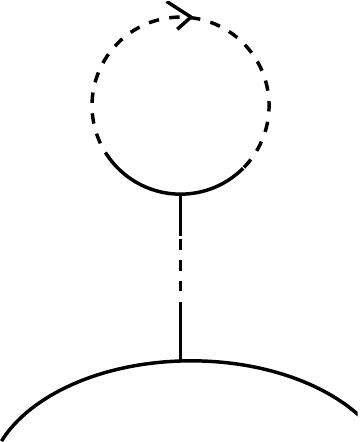}\end{minipage}=\begin{minipage}{2cm}\includegraphics[width=2cm]{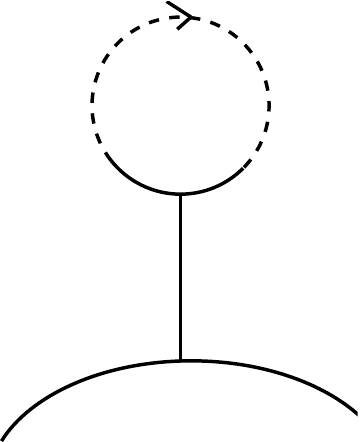}\end{minipage}$$
\item $$\partial \begin{minipage}{2.5cm}\includegraphics[width=2.5cm]{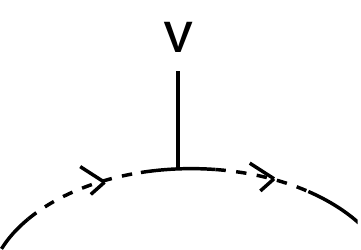}\end{minipage}=\begin{minipage}{2.5cm}\includegraphics[width=2.5cm]{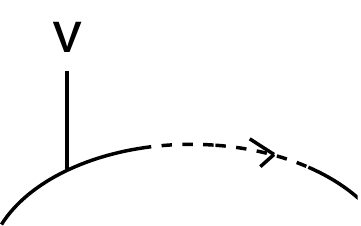}\end{minipage}-\begin{minipage}{2.5cm}\includegraphics[width=2.5cm]{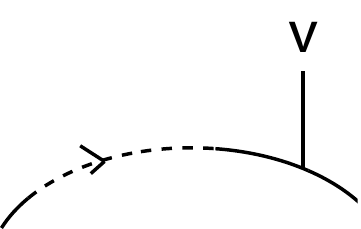}\end{minipage}$$
\item $$\partial \begin{minipage}{2cm}\includegraphics[width=2cm]{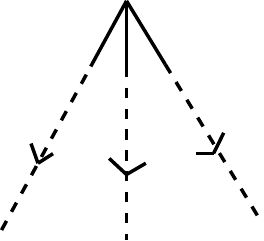}\end{minipage}=\begin{minipage}{2cm}\includegraphics[width=2cm]{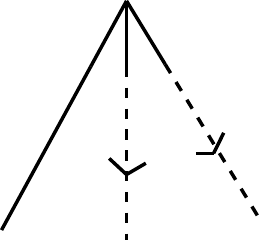}\end{minipage}+\begin{minipage}{2cm}\includegraphics[width=2cm]{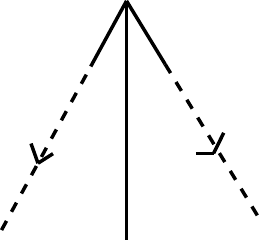}\end{minipage}
+\begin{minipage}{2cm}\includegraphics[width=2cm]{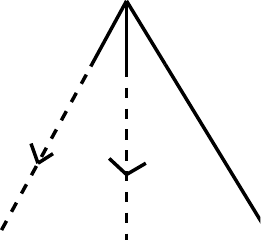}\end{minipage}$$
\end{enumerate}
\caption{Relations in $\gspace(H)$.}\label{boundaries}
\end{figure}

Now we have all the necessary definitions to define the new trace map:

\begin{definition}
Define $\newtrace\colon \cT(H)\to \gspace(H)$ by the composition 
$$
\xymatrix{
\cT(H)\ar@/_2pc/[rr]^{\newtrace}\ar[r]^{\Tr^{\mathrm{CKV}}} & C_1\hairy(H)\ar@{->>}[r] & \gspace(H)
}
$$
\end{definition}

Next we show that $\newtrace$ is well-defined on the cokernel of the Johnson homomorphism.
\begin{theorem}
$\newtrace$ vanishes on the image of the Johnson homomorphism in orders $\geq 2$.
\end{theorem}
\begin{proof}
By Hain's theorem, it suffices to show that $\newtrace([t, X])=0$ if $t$ is of order $1$ and $\newtrace(X)=0$. Indeed, we claim the formula $$\newtrace[t,X]=[t,\newtrace(X)]+[\newtrace(t),X]$$ holds. Assume $t$ and $X$ are single trees. The terms of $\newtrace [t,X]$ come in two types. Those where the added external edges do not join $t$ and $X$ and those where 1 or 2 edges join $t$ and $X$. In the former case, we get the $[t,\newtrace(X)]+[\newtrace(t),X]$ part we are interested in. If one edge joins $t$ and $X$, we have the situation depicted in Figure~\ref{proof} (1). After applying the trace map, the two indicated terms differ by sliding a hair over an edge, so cancel in $\gspace(H)$.
If two hairs join, we have the situation depicted in Figure~\ref{proof} (2), which yields the third $\partial(S_2)$ relation. 

\begin{figure}
\begin{enumerate}
\item\begin{align*}
\left[\begin{minipage}{1.2cm}\includegraphics[width=1.2cm]{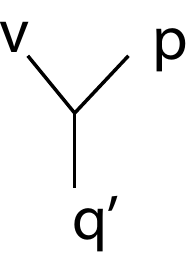}\end{minipage},\begin{minipage}{2.5cm}\includegraphics[width=2.5cm]{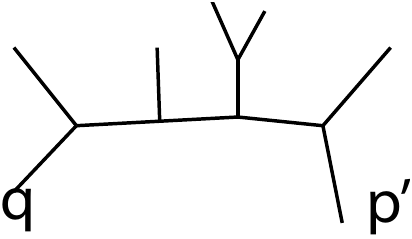}\end{minipage}\right]&=\begin{minipage}{2.5cm}\includegraphics[width=2.5cm]{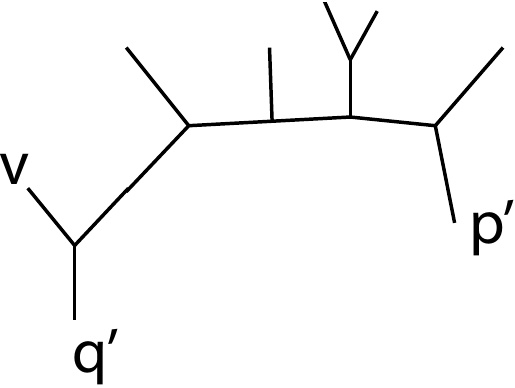}\end{minipage}-\begin{minipage}{2.5cm}\includegraphics[width=2.5cm]{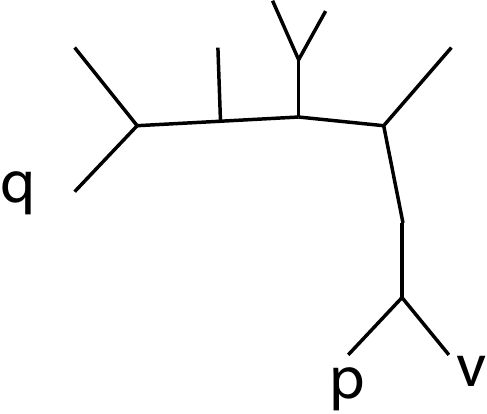}\end{minipage}+\cdots\\
&\overset{\newtrace}{\longrightarrow}\begin{minipage}{2.5cm}\includegraphics[width=2.5cm]{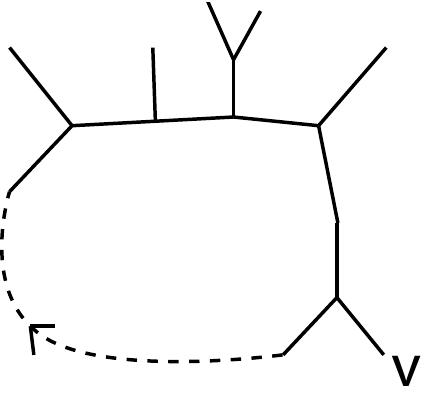}\end{minipage}-\begin{minipage}{2.7cm}\includegraphics[width=2.7cm]{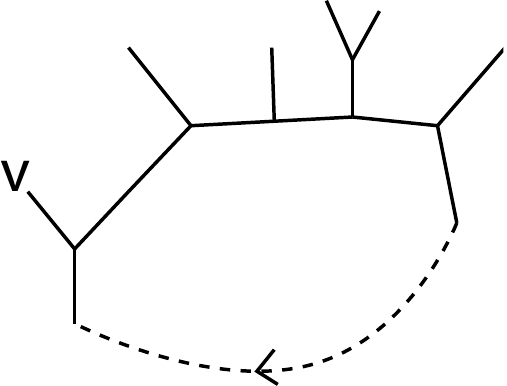}\end{minipage}+\cdots\\
&=\partial(S_2)+\cdots
\end{align*}
\item
\begin{align*}
\newtrace\left[\begin{minipage}{1.2cm}\includegraphics[width=1.2cm]{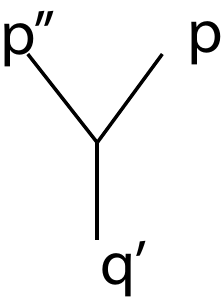}\end{minipage},\begin{minipage}{2.5cm}\includegraphics[width=2.5cm]{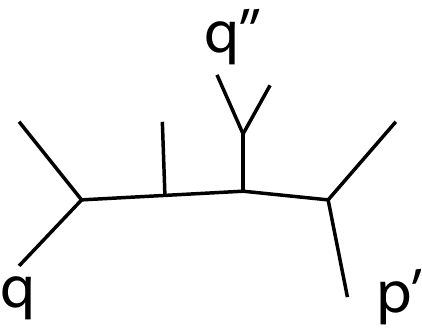}\end{minipage}\right]&=-\begin{minipage}{2.5cm}\includegraphics[width=2.5cm]{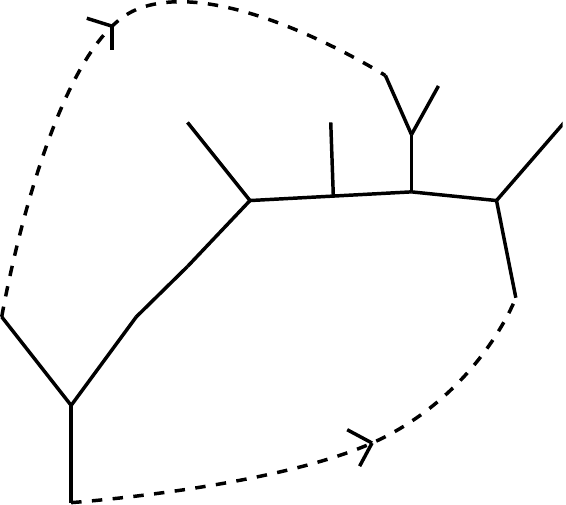}\end{minipage}-\begin{minipage}{2.5cm}\includegraphics[width=2.5cm]{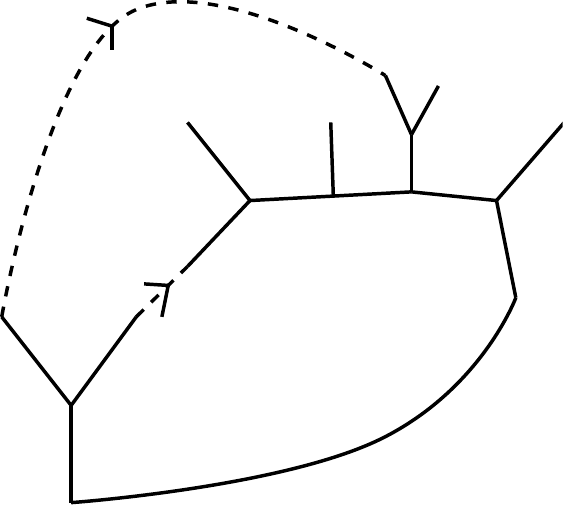}\end{minipage}-\begin{minipage}{2.5cm}\includegraphics[width=2.5cm]{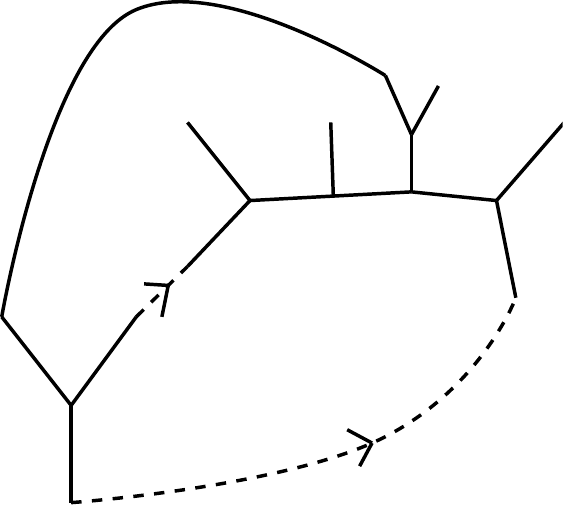}\end{minipage}+\cdots\\
&=\partial(S_2)+\cdots
\end{align*}
\end{enumerate}
\caption{Parts of $\newtrace[t,X]$ in $\partial(S_2)$.}\label{proof}
\end{figure}
So we have shown that $\newtrace[t,X]=[t,\newtrace(X)]+[\newtrace(t),X]$. 
 Now $\newtrace(t)$ is equal to $t$ plus terms where one edge is added. The $t$ is in $\iota(\cT(H))$ and therefore is zero. The second type of term is the first kind of $\partial (S_2)$ relation, so is zero.Thus $\newtrace[t,X]=[t,\newtrace X]$, which inductively shows that $\newtrace$ vanishes on iterated brackets of order $1$ elements. \end{proof}

\section{Comparison to the ES-trace}
The space of connected hairy graphs is graded by the first Betti number (rank) and also by number of hairs. Let $C_{1,r,s}\hairy(H)\subset C_1\hairy(H)$ and $S_{2,r,s}\subset S_2$ be the respective subspaces generated by graphs with $b_1=r$ and $s$ hairs. Define $\gspace_{s,r}(H)=C_{1,r,s}\hairy(H)/\partial S_{2,r,s}$. Then $$\gspace(H)=\bigoplus_{s\geq 0,r\geq 1}\gspace_{s,r}(H)$$

In the next theorem we identify $\gspace_{s,1}(H)$ with the target of the Enomoto-Satoh trace.
\begin{theorem}
There is an isomorphism $\gspace_{s,1}(H)\cong [H^{\otimes s}]_{D_{2s}}$ for $s>1$.
\end{theorem}
\begin{proof}
Notice that $C_{1,1,s}\hairy(H)$ is spanned by trees with two univalent vertices joined by an external edge. Using IHX relations, one gets a loop with s labeled hairs attached. Thus  $C_{1,1,s}\hairy(H)\cong [H^{\otimes s}]_{\Z_2}$ where the $\Z_2$ acts by reflecting the loop, and has sign $(-1)^{s+1}$. So it gives $v_1\otimes\cdots\otimes v_s\mapsto (-1)^{s+1}v_s\otimes\cdots\otimes v_1$. The slide relations have the effect: $v_1\otimes\cdots\otimes v_s=v_s\otimes v_1\otimes \cdots\otimes v_{s-1}$, giving us $[H^{\otimes s}]_{D_{2s}}$. The loop relation is a consequence of IHX and slide relations if $s>1$:
$$
\begin{minipage}{1.7cm}\includegraphics[width=1.7cm]{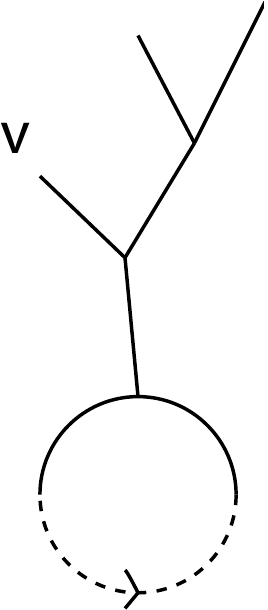}\end{minipage}=\begin{minipage}{2cm}\includegraphics[width=2cm]{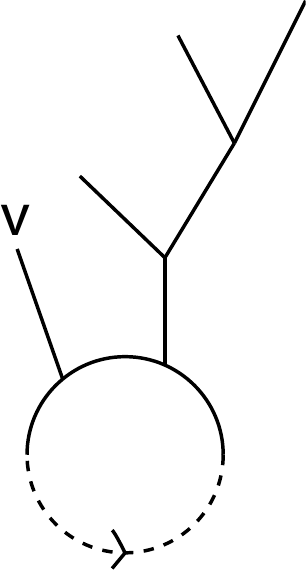}\end{minipage}-\begin{minipage}{1.7cm}\includegraphics[width=1.7cm]{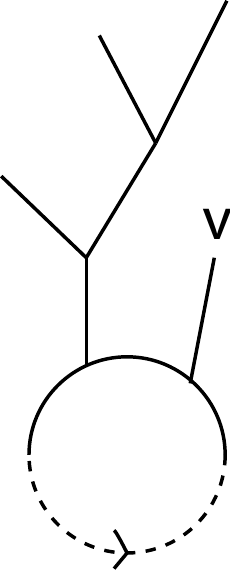}\end{minipage}
$$
Here any tree can, by IHX, be converted into one of the form $[v,X]$, where $v\in H$, so the picture is sufficiently general. Then the last two terms cancel by a slide relation.
\end{proof}

Next we show that $\newtrace$ projected to  $\gspace_{s,1}(H)$ coincides with the ES-trace. First we show that it possesses an additional $\Z_2$-symmetry.
\begin{theorem}\
\begin{enumerate}
\item Let $b\colon H^{\otimes s}\to H^{\otimes s}$ be defined by $b(v_1\otimes\cdots\otimes v_s)=(-1)^{s+1}v_s\otimes\cdots\otimes v_1$. Then
 $\Tr^{\mathrm{ES}}\colon \mathsf D_s(H)\to [H^{\otimes s}]_{\Z_s}$ satisfies $b\Tr^{\mathrm{ES}}=\Tr^{\mathrm{ES}}$. Therefore, without loss of information, $\Tr^{\mathrm{ES}}$ takes values in $[H^{\otimes s}]_{D_{2s}}$.
\item The following diagram commutes:
$$\xymatrix{
\cT_s(H)\ar[r]^{\newtrace}\ar@{>->>}[d]^\eta&\gspace(H)\ar@{->>}[d]\\
\mathsf D_s(H)\ar[r]^{\frac{1}{2}\Tr^{\mathrm{ES}}}&[H^{\otimes s}]_{D_{2s}}
}$$
\end{enumerate}
\end{theorem}
\begin{proof}
We use the isomorphism $\eta\colon \cT_s(H)\to \mathsf D_s(H)$. Let $t\in  \cT_s(H)$ be a labeled tree, and consider $\eta(t)=\sum_x \ell(x)\otimes t_x$. We think of this as a sum of choosing a root for the tree and remembering the label of the root. 
 Satoh's trace map \cite{S} is defined by the embeddings $\mathsf D_s(H)\hookrightarrow \mathsf H\otimes \mathsf L_{s+1}(H)\hookrightarrow H\otimes H^{\otimes{s+1}}$ and then contracting the first two terms to end up in $H^{\otimes s}$. Fix a univalent vertex $x$. Consider what happens if we focus on contracting $\ell(x)$ with a label on a fixed univalent vertex of $t_x$, say $v$. We can rearrange $t_x$ so that $v$ is leftmost, as in the following picture: $$\ell(x)\otimes \begin{minipage}{3.5cm}\includegraphics[width=3.5cm]{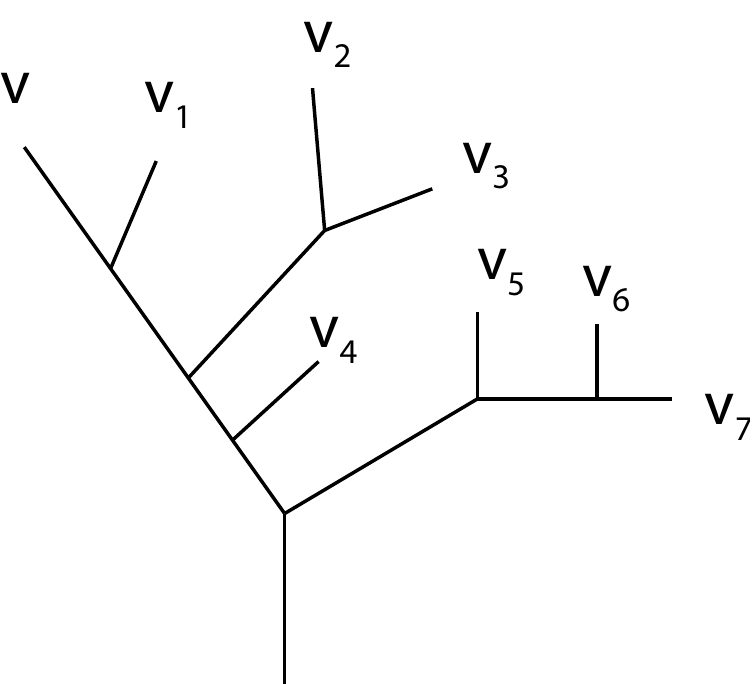}\end{minipage}$$
Since we are concentrating on contracting with $v$, we collect all terms in $H^{\otimes (s+1)}$ where $v$ is first. That means that using the relation $$ \begin{minipage}{.9cm}\includegraphics[width=.9cm]{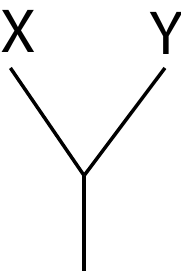}\end{minipage}=X\otimes Y-Y\otimes X$$ the trees growing off of the arc joining $v$ and the root are expanded in the same order they appear. So for example in the picture above we get $\ell(x)\otimes vv_1[v_2,v_3]v_4[v_5[v_6,v_7]]$ which contracts to $\la \ell(x),w\ra v_1[v_2,v_3]v_4[v_5[v_6,v_7]]$. This is the same element of $H^{\otimes s}$ you would get by adding an edge joining $x$ and the vertex labeled $w$ and read off the word around the cycle running along the direction of the added edge, using the fact that IHX relations near the cycle translate to $[X,Y]=XY-YX$.  Thus $\Tr^{\mathrm {ES}}\eta(t)$ can be regarded as summing over adding a directed edge between two leaves of the tree, and reading off the resulting word as you run around the cycle. The extra $\Z_2$ symmetry comes from the fact that you join two vertices once by an edge running in one direction and once with an edge running in the opposite direction. This reverses the word, and yields a sign of $(-1)^{s+1}$. (One sign for flipping the order of contraction, and $s$ signs for the $s$ trivalent vertices of the tree.) This discussion also shows that $\Tr^{\mathrm {ES}}\eta(t)$ is the same as the $1$-edge part of $\newtrace$. The factor of two arises because we only add one edge for every pair of vertices instead of $2$.
\end{proof}

\section{Surjectivity onto a large submodule of $\gspace(H)$}
We begin by defining an analogue of the hairy graph complex and target space $\gspace(H)$ where there is a given bijection from the hairs to $\{1,\ldots,s \}$ as opposed to a labeling of the hairs by vectors.
\begin{definition}\
\begin{enumerate}
\item Let $C_k\hairy[s]$ be the space defined analogously to $C_k\hairy(H)$, but instead of labeling the hairs by vectors in $H$, there are $s$ hairs and a fixed bijection from these hairs to $1,\ldots,s$. The relations are all the same, except there is no multilinearity. Then $C_k\hairy[s]$ is an $\Sym{s}$-module.
\item Similarly define $S_2[s]\subset C_2\hairy[s]$ to be spanned by tripods connected to another tree, by two or three hairs, as well as tripod with a self-loop connected to a tree.
\item $\gspace[s]$ is defined to be $C_1\hairy[s]/(\partial S_2[s]+(\text{trees with no external edges}))$.
\end{enumerate}
\end{definition}
Notice that we have $C_k\hairy[s]\otimes_{\Sym{s}}H^{\otimes s}=\oplus_r C_{k,r,s}\hairy(H)$, and $\gspace[s]\otimes_{\Sym{s}}H^{\otimes s} =\oplus_{r}\gspace_{s,r}(H)$. 

Recall that $H^{\la s\ra}\subset H^{\otimes s}$ is the intersection of the kernels of all pairwise contractions $H^{\otimes s}\to H^{\otimes(s-2)}$. By \cite{FH}, given any partition $\lambda$ of $s$, we have 
\begin{remark}\
\begin{enumerate}
\item $[\lambda]_{\sym s}\otimes_{\sym{s}}H^{\otimes s}\cong [\lambda]_{\GL}$
\item $[\lambda]_{\sym{s}}\otimes_{\sym{s}}H^{\la s\ra}\cong [\lambda]_{\SP}$
\end{enumerate}
\end{remark}
for $\dim(H)$ large enough compared to $s$. 
\begin{definition}
Define a new complex $$C_{k}\hairy{\la H\ra}=\bigoplus_s C_{k}\hairy[s]\otimes_{\sym s}H^{\la s\ra},$$ and a new space $$\gspace\la H\ra =\bigoplus_s \gspace[s]\otimes_{\sym s} H^{\la s\ra}.$$
 \end{definition}
By \cite{FH}, $H^{\otimes s}$ decomposes as a direct sum of $\SP$-modules, including $H^{\la s\ra}$, in a natural way, so there is a projection $H^{\otimes s}\to H^{\la s\ra}$. This gives projections $\pi\colon C_k\hairy(H)\twoheadrightarrow C_k\hairy{\la h\ra}$ and $\pi\colon\gspace(H)\twoheadrightarrow \gspace\la H\ra$.

The following theorem is a consequence of a more general theorem of \cite{CKV2}.
\begin{theorem}[Conant-Kassabov-Vogtmann]\label{ckvthm}
For $\dim H$ large enough compared to $s$,
$$\pi\circ \Tr^{\mathrm{CKV}}\colon  \cT_s(H)\to \bigoplus_r C_{1,r,s}\hairy{\la H\ra}$$
is an isomorphism.
\end{theorem}

\begin{corollary}
The composition $\pi\circ\newtrace\colon\cT_s(H)\to \gspace_s\la H\ra$ is an epimorphism.
\end{corollary}
\begin{proof}
Consult the following commutative diagram
$$\xymatrix{
\cT_s(H)\ar@/^2pc/@{>->>}[rr]^{\cong (\text{Thm. \ref{ckvthm}})}
\ar[r]^{\Tr^{\mathrm{CKV}}}&\bigoplus_r C_{1,r,s}\hairy(H)\ar@{->>}[d]\ar@{->>}[r]^{\pi}&\bigoplus_r C_{1,r,s}\la H\ra\ar@{->>}[d]\\
&\bigoplus_r \gspace_{s,r}(H)\ar@{->>}[r]^{\pi}&\gspace\la H\ra
}
$$
\end{proof}

\begin{corollary} In particular $\Tr^{\mathrm{ES}}$ surjects onto $\gspace_{s,1}\la V\ra \cong [H^{\la s\ra}]_{D_{2s}}$.
\end{corollary}
 Also note that by the above remark if $\gspace_{s,r}(H)=\oplus_\lambda m_\lambda[\lambda]_{\GL}$, then $\gspace_{s,r}\la H\ra=\oplus_\lambda m_\lambda[\lambda]_{\SP}$, so  the $\GL(H)$-representation theory for $\gspace(H)$ determines the $\SP(H)$ representation theory for $\gspace\la H\ra$.

\section{Presentation for $\gspace_{s,2}(H)$}\label{sec:2loop}
%The space $\gspace_{1,s}(H)$ was already determined to be $[H^{\otimes s}]_{D_{2s}}$.
 To set up the main theorem of this section let $T(H)$ be the tensor algebra and $T^+(H)$ the positive degree part of it. Define an involution $\rho\colon T(H)\to T(H)$ by  $\rho(v_1\cdots v_k)=\overline{v_1\cdots v_k}= (-1)^kv_k\cdots v_1$. For an index set $I=\{i_1,\ldots, i_k\}$, let $v_I=v_{i_1}\cdots v_{i_k}$.
Now define a
coproduct $\Delta\colon T(H)\to T(H)\otimes T(H)$ by $$\Delta(v_K)=\sum_{K=I\cup J}{v_I}\otimes v_J$$ where the sum is over all partitions of $K$ into two disjoint sets $I$ and $J$.

In this section we prove the following theorem:
\begin{theorem}
$$\bigoplus_{s\geq 0}\gspace_{s,2}(H)\cong [T^+(H)\otimes T^+(H)]_{\Z_2\times \Z_2}/\mathrm{Rel}$$
where the $\Z_2\times \Z_2$ acts via $v_I\otimes w_J\mapsto \overline{v_I}\otimes \overline{w_J}$ and $v_I\otimes w_J\mapsto \overline{w_I}\otimes \overline{v_J}$. The relations $\mathrm{Rel}$ are of the form
\begin{enumerate}
\item $-(v_0\otimes 1+1\otimes v_0)v_I\otimes w_J+v_I\otimes w_J(v_0\otimes 1+1\otimes v_0)=0$ where $v_0\in H$.
\item $(\rho\otimes 1)\left(\Delta(v_I)(1\otimes w_J)\right)+v_I\otimes w_J+(1\otimes \rho)\left((v_I\otimes 1)\Delta(w_J)\right)=0$.
\end{enumerate}
\end{theorem}
\begin{proof}
As in the case of $\gspace_{s,1}$ we can apply IHX relations so that we have a trivalent core graph with hair attached. So we have a unitrivalent tree with all of its univalent vertices joined by external edges in pairs,  and to which $s$ hairs are attached. By IHX relations we can move the hair to the edges of the tree that attach to the external edges, and by slide relations we can assume that the hairs are all attached on one side of the external edge. Thus we have two types of generators as depicted in Figure~\ref{generators}. The subscript $e$ stands for ``eyeglasses" and the subscript $t$ stands for ``theta."

\begin{figure}
\begin{enumerate}
\item
$$[v_1\cdots v_m|w_1\cdots w_n]_e=\begin{minipage}{5cm}
\includegraphics[width=5cm]{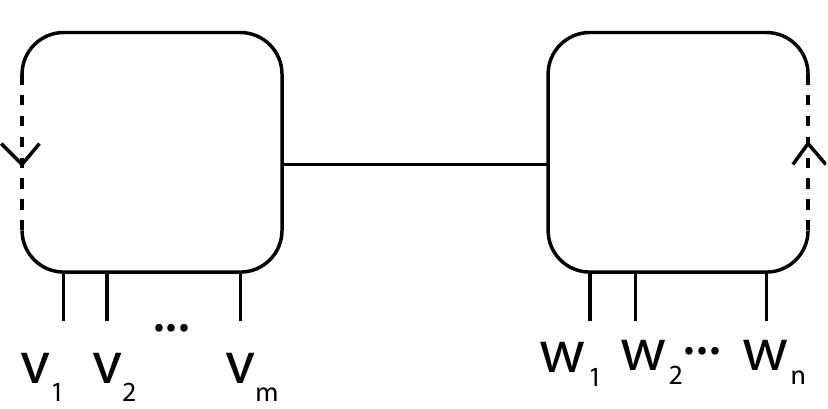}
\end{minipage}$$
\item
$$[v_1\cdots v_m|w_1\cdots w_n]_t=\begin{minipage}{4cm}
\includegraphics[width=4cm]{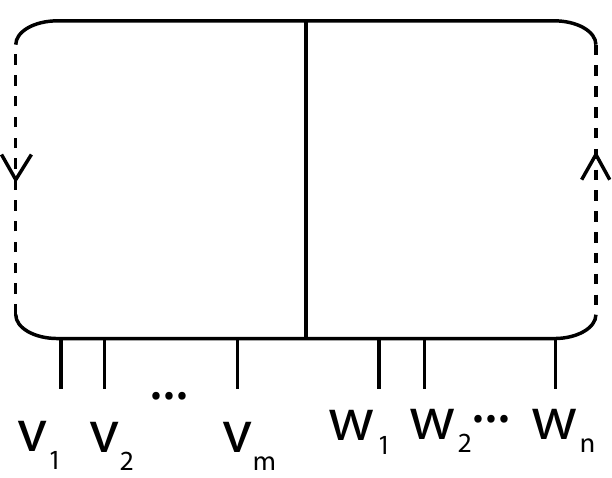}
\end{minipage}$$
\end{enumerate}
\caption{Generators of $\gspace_{s,2}$ where $m+n=s$.}\label{generators}
\end{figure}

By multilinearity, we may extend the symbols $[x|y]_{e,t}$ to any $x,y$ in the tensor algebra $T(H)$. Symmetries of the graphs give rise to the relations, using the sliding relations to move hairs back to the bottom of the picture:
\begin{enumerate}
\item[S1]: $[v_I|w_J]_e=[\overline{v_I}|w_J]_e$.
\item[S2]: $[v_I|w_J]_e=[\bar w_J| \bar v_I]_e$
\item[S3]: $[v_I|w_J]_t=[\bar v_I| \bar w_J]_t$
\item[S4]: $[v_I|w_J]_t=[\bar w_J| \bar v_I]_t$
\end{enumerate}
The loop relation gives us (using IHX)
\begin{enumerate}
\item[L]: $[\,|w_J]_t=[\,|w_J]_e=0$
\end{enumerate}
The IHX relation has two effects. IHX1 relates the theta graph and eyeglass graph . However, we also used IHX to push hairs to be near the external edge, and the ambiguity of where to push a hair labeled $v_0$ gives IHX1 below. 
\begin{enumerate}
\item[IHX1]: $[v_Iv_0|w_J]-[v_I|v_0w_J]-[v_0v_I|w_J]+[v_I|w_Jv_0]=0$ (e or t) $\deg(v_0)=1$
\item[IHX2]: $[v_I|w_J]_e=[v_I|w_J]_t+[\bar v_I|w_J]_t$
\end{enumerate}
Finally the boundary of a tripod with three incident edges yields
\begin{enumerate}
\item[TRI] : Then $\sum_{I\cup J=K}[ \bar{v_I}|v_J w_L]_t+[v_K|w_L]_t+\sum_{I\cup J=L}[v_K\bar w_I|w_J]_t=0$.
\end{enumerate}
To see this consider Figure~\ref{boundarytripod}. A boundary is shown in (1). To move the hair off of the left edge of the first summand, we repeatedly use the IHX relation shown in (2), to iteratively build up the terms described in (3). 
\begin{figure}
\begin{enumerate}
\item
$$
\partial\begin{minipage}{3cm}\includegraphics[width=3cm]{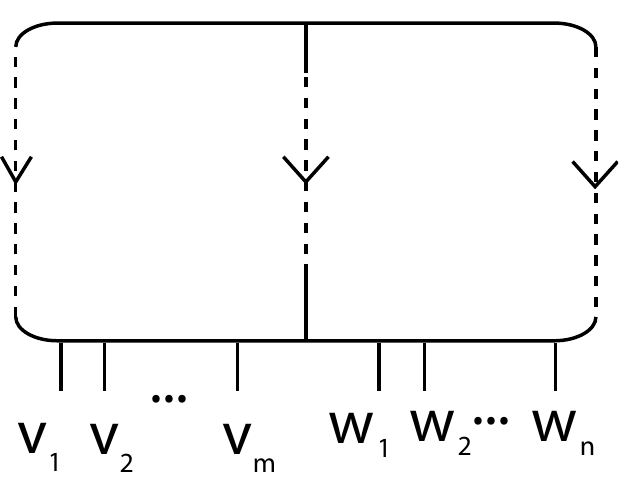}\end{minipage}=
\begin{minipage}{3cm}\includegraphics[width=3cm]{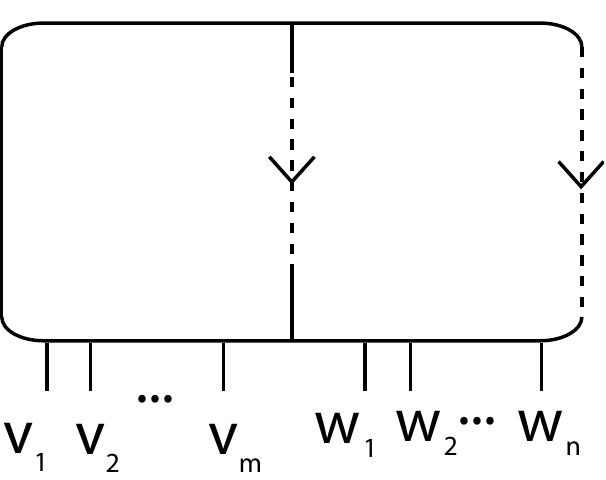}\end{minipage}+
\begin{minipage}{3cm}\includegraphics[width=3cm]{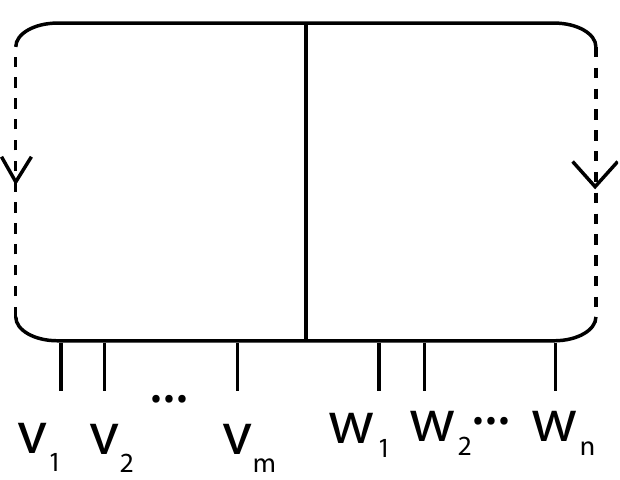}\end{minipage}+
\begin{minipage}{3cm}\includegraphics[width=3cm]{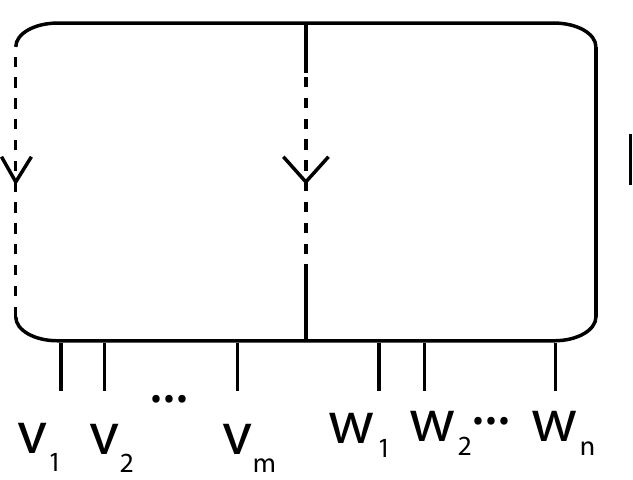}\end{minipage}
$$
\item $$
\begin{minipage}{2cm}\includegraphics[width=2cm]{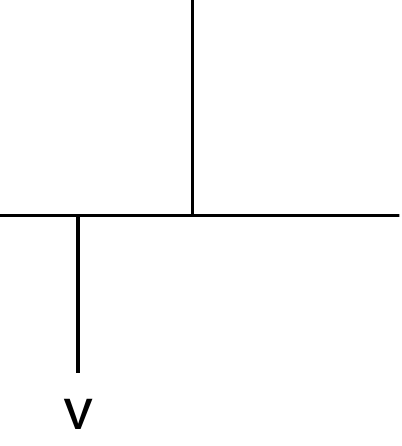}\end{minipage}=\begin{minipage}{2cm}\includegraphics[width=2cm]{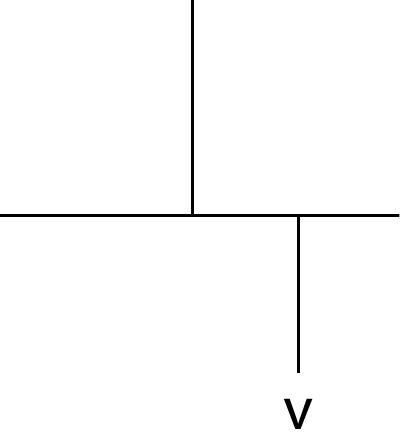}\end{minipage}-\begin{minipage}{2cm}\includegraphics[width=2cm]{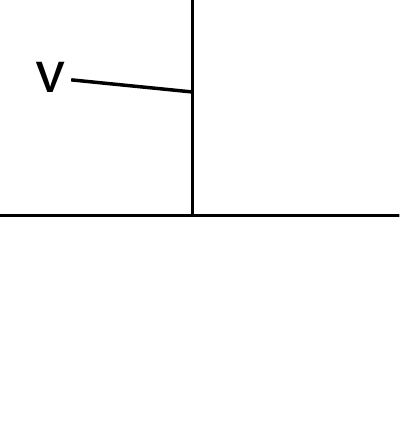}\end{minipage}
$$
\item $$\begin{minipage}{3cm}\includegraphics[width=3cm]{tribdry2.pdf}\end{minipage}=-\sum_{I\cup J=\{1,\ldots, m\}}\begin{minipage}{3.3cm}\includegraphics[width=3.3cm]{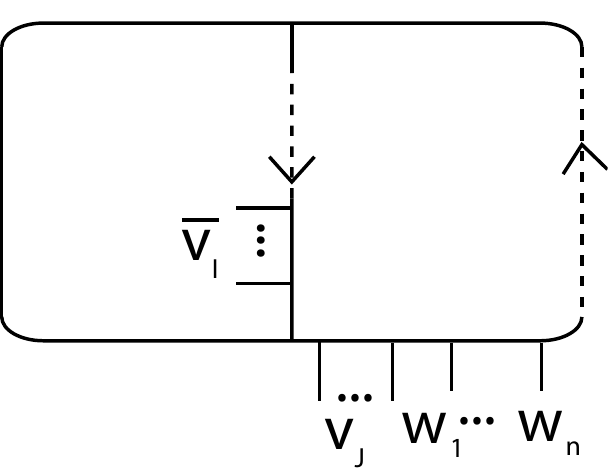}\end{minipage}=-\sum_{I\cup J=\{1,\ldots, m\}}[\overline{v_I}|v_Jw_1\cdots w_n].$$
$$\begin{minipage}{3cm}\includegraphics[width=3cm]{tribdry3.pdf}\end{minipage}=-[v_1\cdots v_m|w_1\cdots w_n].$$
$$
\begin{minipage}{3cm}\includegraphics[width=3cm]{tribdry4.pdf}\end{minipage}-\sum_{I\cup J=\{1,\ldots, n\}}\begin{minipage}{3.3cm}\includegraphics[width=3.3cm]{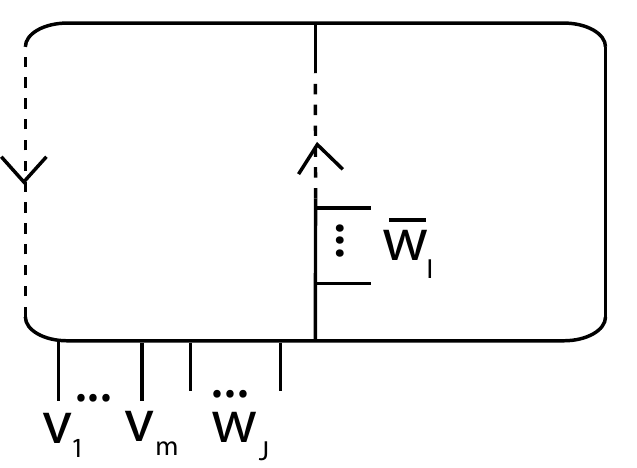}\end{minipage}=-\sum_{I\cup J=\{1,\ldots, n\}}[v_1\cdots v_m w_I|\overline{w_J}].
$$
\end{enumerate}
\caption{Deriving the TRI relation.}\label{boundarytripod}
\end{figure}

Using $IHX2$ we can express everything in terms of the $t$ generators. S1 and S2 are consistent with S3 and S4, so we are left with relations S3,S4,L,IHX1 and TRI. Interpreting $[v_I|w_J]\in T(H)\otimes T(H)$ gives the theorem.
\end{proof}

Computer calculations using this presentation yield the following results:

\begin{theorem}\
\label{thm:omega2}
For $s\leq 5$, $\gspace_{s-2,2}(H)=\gspace_{s-2,2}\la H\ra=0$.
\begin{enumerate}
\item $\gspace_{4,2}\la H\ra \cong  [1^4]_{\SP}\oplus[31]_{\SP}$, yielding representations in $\mathsf C_6$.
\item $\gspace_{5,2}\la H\ra\cong 2 [31^1]_{\SP}\oplus [2^21]_{\SP}\oplus[21^3]_{\SP}$, yielding representations in $\mathsf C_7$.
\item $\gspace_{6,2}\la H\ra\cong [1^6]\oplus 2[51]\oplus 3[42]\oplus[3^2]\oplus 3[321]\oplus 2[2^3]\oplus 2[2^21^2]\oplus 2[21^5]\oplus [1^6]$, yielding representations in $\mathsf C_8$.
\end{enumerate}
\end{theorem}

\section{Representation theory of $[H^{\la s\ra }]_{D_{2s}}$} \label{reptheory}
In this section we analyze the $\SP$-representation theory of $[H^{\la s\ra }]_{D_{2s}}$, which is the same as the $\GL$-representation theory of $[H^{\otimes s}]_{D_{2s}}$, which can be analyzed via classical Schur-Weyl duality and character theory. Hand calculations with characters yield the following results for low $s$.
\begin{theorem}\
\begin{enumerate}
\item $[H^{\la 4\ra }]_{D_{8}}\cong [21^2]_{\SP}$, which picks up the $[21^2]_{\SP}\in \mathsf C_4$ found by Morita.
\item $[H^{\la 5\ra}]_{D_{10}}\cong [5]_{\SP}\oplus [32]_{\SP}\oplus [2^21]_{\SP}\oplus [1^5]_{\SP}$. This picks up all of the size $5$ $\SP$-representations in $\mathsf C_5$.
\item $[H^{\la 6\ra }]_{D_{12}}\cong [3^2]_{\SP}\oplus 2[41^2]_{\SP}\oplus [321]_{\SP}\oplus [31^3]_{\SP}\oplus [2^21^2]_{\SP}$. Comparing this to computer calculations of $\mathsf C_6$ due to Morita-Sakasai-Suzuki \cite{MSS3}, this picks up all size $6$ representations in $\mathsf C_6$.
\end{enumerate}
\end{theorem}
These calculations are suggestive of the following (somewhat optimistic) conjecture:
\begin{conjecture}
All representations of size $s$ in $\mathsf C_{s}$ are contained in $[H^{\la s\ra}]_{D_{2s}}$.
\end{conjecture}

In the next theorem we analyze the $4$ representations of lowest complexity, showing that we pick up the Enomoto-Satoh and Morita representations.

\begin{theorem}\label{thm:ESM}\
\begin{enumerate}
\item The representations $[1^s]_{\SP}$ only occur when $s=4m+1$, and in that case with multiplicity one. These are the Enomoto-Satoh terms contained in $[H^{\la 4m+1\ra}]_{D_{2(4m+1)}}$. 
\item The representations $[s]_{\SP}$ only occur when $s=2m+1$, and in that case with multiplicity one. These are the
 Morita terms contained in $[H^{\la 2m+1\ra}]_{D_{2(2m+1)}}$.
\item The representations $[s-1,1]_{\SP}$ and $[2,1^{s-2}]_{\SP}$ do not occur in $[H^{\la s\ra} ]_{D_{2s}}$.
\end{enumerate}
\end{theorem}
\begin{proof}
For the first statement, it suffices to examine the multiplicity of $[1^{s}]_{\GL}$ contained in $[H^{\otimes(s)}]_{D_{2(s)}}$. 
Let $a,b$ be generators of $D_{2s}$. Then 
$$a\cdot(x_1\wedge\cdots\wedge x_s)=x_2\wedge \cdots\wedge x_s\wedge x_1=(-1)^{ s-1}x_1\wedge\cdots\wedge x_s$$
$$b\cdot(x_1\wedge\cdots\wedge x_s)=(-1)^{s+1}x_s\wedge\cdots\wedge x_1=(-1)^{s+1+\lfloor s/2\rfloor}x_1\wedge\cdots\wedge x_s$$
So we need $s-1$ and $s+1+\lfloor s/2\rfloor$ both even, which occurs if and only if $s=4m+1$.

The second statement is proven similarly.

For the third statement, one considers the exact sequences $$0\to [s-1,1]_{\GL}\to S^{s-1}(H)\otimes H\to S^{s}(H)\to 0$$ and
$$0\to [2,1^{s-2}]_{\GL}\to \ext^{s-1}(H)\otimes H\to \ext^{s}(H)\to 0,$$
checking that the $D_{2s}$ coinvariants of $S^{s-1}(H)\otimes H$ and $\ext^{s-1}(H)\otimes H$ coincide with those of $S^{s}(H)$ and $\ext^{s}(H)$ respectively.
\end{proof}

Next we prove a convenient proposition which is instrumental in calculating the $D_{2s}$ coinvariants of a representation $[\lambda]_{D_{2s}}$.

\begin{proposition}
In the untwisted case, the coinvariants $([\lambda]_{\sym{s}})_{D_{2s}}$ have dimension 
$$\frac{1}{2s}\sum_{g\in D_{2s}}\chi_{\lambda}(g),$$ where $\chi_\lambda$ is the character for $[\lambda]_{\sym{s}}$.
In the case where $D_{2s}$ acts with the $\Z_2$ twist, the dimension is
$$\frac{1}{2s}\sum_{g\in D_{2s}}\sigma(g)\chi_{\lambda}(g),$$
where $\sigma\colon D_{2s}\to  \{\pm 1\}$ maps $a\mapsto 1, b\mapsto -1$.
\end{proposition}
\begin{proof}
Given a character $\chi$ for the dihedral group, define $\int \chi =\frac{1}{2s}\sum_{g\in D_{2s}}\chi(g)$. Consulting the character tables for the dihedral group (see \cite{JL} section 18.3), for each irreducible character $\chi$, we have 
$$\int \chi=\begin{cases}
1&\text{if } \chi\text{ is the character for the trivial representation}\\
0&\text{otherwise}
\end{cases}
$$
So decomposing $[\lambda]_{\sym s}$ as a direct sum of irreducible $D_{2s}$-modules, and writing the character $\chi_{\lambda}$ as a sum of the corresponding dihedral characters, the result follows. The twisted case follows by a similar analysis.
\end{proof}

It is a remarkable fact that for symmetric group elements $\sigma$ with large support, $\chi_\lambda(\sigma)\ll\chi_\lambda(1)$  (see e.g. \cite{roichman,LS}). Since elements of the dihedral group fix at most two points, this implies that the multiplicities of the $D_{2n}$ coinvariants appearing in the previous proposition are approximately $\frac{1}{2n}\chi_\lambda(1)=\frac{1}{2n}\dim([\lambda]_{\sym n})$. For ``most" $\lambda$, we have $\dim [\lambda]_{\sym n}\gg 2n$, and so for such representations $[\lambda]_{\sym n}$ appears in $[H^{\la n\ra}]_{D_{2n}}$ and thus in $\mathsf C_n$.
This heuristic argument can be made precise by examining the actual constants involved in the estimates,  constructing infinite families of nonzero representations.

As an exercise we work out the exact multiplicities in a couple of different cases.

\begin{theorem}\label{thm:p}
Let $p\geq 3$ be prime. 
Let $\alpha_k=\binom{p}{k}-\binom{p}{k-1}$.
If $k>1$ is odd then $[k,p-k]_{\SP}$ appears with multiplicity $\frac{\alpha}{2p}$ in $\mathsf C_p$. If $k=2m$, let $\beta_m=\binom{(p-1)/2}{m}-\binom{(p-1)/2}{m-1}$. Then $[k,p-k]_{\SP}$ appears with multiplicity $\frac{\alpha_{2m}+\beta_m}{2}$ in $\mathsf C_p$.
\end{theorem}
\begin{proof}
Given the partition $\lambda=(k,p-k)$, it is easy to calculate $\int \chi$ using the Frobenius character formula. The values of the character on the conjugacy classes $1,a^r,b$ are as follows:
\begin{align*}
\chi_{\lambda}(1)&=\binom{p}{k}-\binom{p}{k-1}\\
\chi_\lambda(a^r)&=\begin{cases}-1&k=1\\0&k\geq 2\end{cases}\\
\chi_{\lambda}(b)&=\begin{cases}
0&k\text{ odd}\\
\binom{(p-1)/2}{m}-\binom{(p-1)/2}{m-1}&k=2m
\end{cases}
\end{align*}
Then $\int\chi_\lambda=\frac{1}{2p}(\chi_{\lambda}(1)+(p-1)\chi_\lambda(a^r)+p\chi_\lambda(b))$, which yields the multiplicities stated in the theorem.

\end{proof}
In the next theorem, we consider order $2p$ where $p$ is prime in order to pick up some even order representations. Again, for simplicity we restrict to $2$ rows. 

\begin{theorem}\label{thm:2p}
Let $p\geq 3$ be prime. 
 For $1< k\leq p$, the representation $[2p-k,k]_{\SP}$ appears in  $\mathsf C_{2p}$ with multiplicity 
$$\frac{1}{4p}\left[\binom{2p}{k}-\binom{2p}{k-1}+(-1)^k(p+1)\binom{p}{m}-p\binom{p-2}{m}+p\binom{p-2}{m-1}+2(p-1)\delta_{p,k}\right],$$
where $m=\lfloor(k/2)\rfloor$, and $\delta_{p,k}$ is equal to $0$ unless $p=k$, in which case it is $1$.
\end{theorem}
\begin{proof}
As in the proof of the previous theorem, we calculate $\frac{1}{2(2p)}\sum_{g\in D_{2(2p)}}\sigma(g)\chi_\lambda(g)$. The conjugacy classes for  $D_{2p}$ and their sizes are written down in in Figure~\ref{chartable}. The dimensions of the $D_{2(2p)}$ coinvariants are then $$\frac{1}{4p}\left(\chi_\lambda(1)+\chi_\lambda(a^p)+(p-1)\chi_\lambda(a^{2r+1})+(p-1)\chi_{\lambda}(a^{2r})-p\chi_\lambda(b)-p\chi_\lambda(ab)\right).$$
 On the symmetric group side, we need to compute $\chi_{\lambda}$ for conjugacy classes of $1,a,a^2,a^p,b,ab$ where $1$ has $2p$ fixed points, $a$ has 1 $2p$-cycle, $a^2$ has 2 $p$-cycles, $a^p$ has $p$ 2-cycles, $b$ has $p$ $2$-cycles and $ab$ has $p-2$ 2-cycles and $2$ fixed points. Using the Frobenius character formula one achieves the values listed in the chart.  
\begin{figure}
\begin{center}
\begin{tabular}{c|cccccc}
elt. of $D_{2(2p)}$&1&$a^p$&$a^r$, $r\text{ odd}$&$a^r, r\text{ even}$&$b$&$ab$\\
\hline
size of conj. class&$1$&$1$&$p-1$&$p-1$&$p$&$p$\\
\hline
$\chi_{[2p-1,1]}$&$2p-1$&$-1$&$-1$&$-1$&$-1$&$1$\\
$\chi_{[2p-2m,2m]}$&$\binom{2p}{2m}-\binom{2p}{2m-1}$&$\binom{p}{m}$&$0$&$0$&$\binom{p}{m}$&$\binom{p-2}{m}-\binom{p-2}{m-1}$\\
$\chi_{[2p-2m-1,2m+1]}$&$\binom{2p}{2m+1}-\binom{2p}{2m}$&$-\binom{p}{m}$&$0$&$0$&$-\binom{p}{m}$&$\binom{p-2}{m}-\binom{p-2}{m-1}$\\
$\chi_{[p,p]}$ &$\binom{2p}{p}-\binom{2p}{p-1}$&$-\binom{p}{m}$&$0$&$2$&$-\binom{p}{m}$&$\binom{p-2}{m}-\binom{p-2}{m-1}$
\end{tabular}
\end{center}
\caption{Characters for $[2p-k,k]_{\mathbb S_p}$ evaluated on conjugacy classes of $D_{2(2p)}$. In the last row, suppose $p=2m+1$.}\label{chartable}
\end{figure}

\end{proof}

\end{document}